\documentclass[a4paper, 12pt]{amsart}
\usepackage{amssymb,amscd}
\usepackage{fullpage}
\newtheorem{thm}{Theorem}[section]
\newtheorem{prop}[thm]{Proposition}
\newtheorem{lem}[thm]{{Lemma}}
\newtheorem{defn}[thm]{Definition}
\newtheorem{example}[thm]{Example}

\newtheorem{cor}[thm]{Corollary}

\numberwithin{equation}{section}

\def\.{\cdot}
\def\<{\left\langle}
\def\>{\right\rangle}
\def\({\left(}
\def\){\right)}

\renewcommand{\phi}{\varphi}

\def\L{\mathbb\L}

\def\subset{\subseteq}

\def\epsilon{\varepsilon}

\def\ker{\operatorname{ker}}

\def\R{\mathbb R}

\def\virt-dim{\operatorname{virt-dim}}

\def\Int{\operatorname{Int}}
\def\cl{\operatorname{Cl}}

\title[On Generalized Topological Groups]
{On Generalized Topological Groups}
\author{Murad Hussain \& Moiz ud Din Khan \&  Cenap \"{O}zel}
\address{COMSATS Institute of Information Technology (CIIT), Islamabad-Pakistan }
\email{murad@comsats.edu.pk}
\address{COMSATS Institute of Information Technology (CIIT), Islamabad-Pakistan }
\email{moiz@comsats.edu.pk}
\address{AIBU Golkoy Kampusu, Bolu 14280, Turkey.}
\email{cenap@ibu.edu.tr}
\thanks{The last author is indebted to Higher Education Commission of Pakistan and COMSATS Institute of Information Technology
for their financial support and hospitality during the last author's visit to the Department of Mathematics at CIIT. }
\subjclass{General topology, Topological groups}
\begin{document}
\begin{abstract}
In this work, we will introduce the notion of generalized topological groups using generalized topological structure and generalized continuity defined by \'{A}. Cs\'{a}sz\'{a}r \cite{csaszar11}.  We will discuss some basic properties of this kind of structures and connectedness properties of this structures are given.

{\bf Keywords:} Generalized topology; generalized continuity; generalized topological groups; generalized connectedness.

\end{abstract}

\maketitle
%%%%%%%%%%%%%%%%%%%%
%%%%%%%%%%%%%%%%%%%%%%%%%%%%%%%%%%%%%%%%%%
%                 Introduction
%%%%%%%%%%%%%%%%%%%%%%%%%%%55
%%%%%%%%%%%%%%%%%%%%%%%%%%%%%%%%%%%%%%%%%%%%%%%%%
\section{Introduction.}
Cs\'{a}sz\'{a}r first introduced generalized topological structures in his paper \cite{csaszar11} and he also studied the notion of generalized continuity.
\begin{defn}
Let $X$ be any set and let $ \mathcal G \subset \mathcal {P} $ be a subfamily. Then $ \mathcal G$ is a {\bf generalized topology} if, $\emptyset \in  \mathcal G$ and for any index set $I$,  $ \cup_{i \in I} O_i \in  \mathcal G$ whenever $ O_i \in  \mathcal G, i \in I.$ A $\mathcal G$-topological space $X$ is called {\bf strong} if the set $X$ itself is a $\mathcal G$-open. We will denote generalized topology by $ \mathcal G$-topology.

The elements of $ \mathcal G$ are called {\bf $ \mathcal {G}$-open sets}. Similarly, a {generalized closed set}, or $ \mathcal {G}$-closed set, is defined as complement of a $ \mathcal {G}$-open set.

Let $X$ and $Y$ be two $ \mathcal G$-topological spaces and let $f:X \rightarrow Y$ be a function. Then $f$ is called {\bf $ \mathcal {G}$-continuous} on $X$ if for any $ \mathcal {G}$-open set $O$ in $Y$, $f^{-1}(O)$ is $ \mathcal {G}$-open in $X$. 

The function $f$ is called a {\bf $ \mathcal G$-homeomorphism} from $X$ to $Y$ if both $f$ and $f^{-1}$ are $ \mathcal G$-continouos functions. If we have a $ \mathcal G$-homeomorphism between $X$ and $Y$ we say that they are {\bf $ \mathcal G$-homeomorphic} and we denote them by $X\cong_{\mathcal G} Y$.   
\end{defn}
Similarly,  Cs\'{a}sz\'{a}r defined neighbourhood systems in the generalized structure \cite{csaszar11}.
We need the following result from \cite{csaszar11} to give the relation between open sets and neighbourhood systems.
\begin{lem}
Let $\Psi$ be a $ \mathcal G$-neighboorhood system on $X$, let $g \subset \mathcal P(X)$ . For  $O \in g$ and $x\in O$ there exists a subset $V \in \psi(x), \psi \in \Psi x \in V \subset O$  iff $g$ is a $ \mathcal G$-topology on $X$. 
\end{lem}

\begin{defn}
Let $X$ be any set and let $\Psi, g$ be $ \mathcal G$-neighbourhopod system and $ \mathcal G$-topology on $X$ respectively. Let $A \subset X$. A point $x\in X$ is called {\bf interior point} of $A$ if there exists a subset $V\in \psi(x), V\subset A$. 
\end{defn}
All interior points of $A\subset X$ is called {\bf interior} of $A$. It is denoted by ${\Int}_{\mathcal G}(A)$. By definition it is obvious that ${\Int}_{\mathcal G}(A) \subset A$. Actually interior of $A, {\Int}_{\mathcal G}(A)$ is equal to union of all $ \mathcal G$-open sets contained in $A$. Similarly we can define {\bf $ \mathcal G$-closure} of $A$ as intersection of all $ \mathcal G$-closd sets containing $A$. It is deonted by $\cl_{\mathcal G}(A)$.

By definition interior of a set is a $ \mathcal G$-open set while closure of a set is a $ \mathcal G$-closed set. 
Now we are ready to give the notion of point-wise continuity using the neighbourhood systems. From \cite{csaszar11}, we have the following definition.
\begin{defn}
Let $(X,g_X,\Psi(X))$ and $ (Y, g_Y,\Psi(Y))$ be two $ \mathcal G$-topological spaces with $ \mathcal G$-neighbourhood systems and $f: X\rightarrow Y$ be a function. We say that $f$ is {\bf point-wise $ \mathcal G$-continuous} at $x\in X$ if for every $ \mathcal G$-neighbourhood $V\in \psi(f(x))$, there exists a $ \mathcal G$-neighbourhood $U\in \psi(x)$ such that $f(U)\subset V$.
\end{defn}
$f$ is {\bf $(\Psi_X, \Psi_Y)$-continuous} on $X$ if $f$ is point-wise continuous at every point $x\in X$.
This continuity is more general than global continuity. 

\smallskip
The fundamental reference for topological groups and their properties is \cite{Arhangel01}.

%%%%%%%%%%%%%%%%%%%%%%%%%%%%%%%%%%%%%%%%%%%%%%%%%%%%%%%%%%%
%%%%%%%%%%%%%%%%%%%%%%%%
%%%%                         Second Section
%%%%%%%%%%%%%%%%%%%%%%%%%%%%
%%%%%%%%%%%%%%%%%%%%%%%%%%%%%%%%%%%%%%%%%%%%%%%%%%%%%%%%%%%

\maketitle

\section{$ \mathcal G$-Topological Groups}
In this section, we will introduce $ \mathcal G$-topological groups and give basic properties of this kind of groups. In order to define this, we need  $T_2$ $ \mathcal G$-topological spaces. From \cite{csaszar13}, we have the following definition.
\begin{defn}
Let $(X,\mathcal G)$ be a $ \mathcal G$-topological space. Then $(X,\mathcal G)$ is called a {\bf $T_2$ space} if for every pair $x,y$ in $X$ with $x\neq y$, there exist disjoint $ \mathcal G$-open sets $U,V$ containing $x,y$ respectively.
\end{defn}
A base of $ \mathcal G$-topological space can be defined as follows.
\begin{defn}
Let $(X,\mathcal G)$ be a $ \mathcal G$-topological space and $\mathcal {B} \subset \mathcal G$. Then $\mathcal {B}$ is called a {\bf base} for the $ \mathcal G$-topology $\mathcal G$ if every $ \mathcal G$-open set can be written as union of some elements of $\mathcal {B}$.
\end{defn}
\begin{defn}
Let $(X, \mathcal G_X)$ and $(Y, \mathcal G_Y)$ be two $ \mathcal G$-topological spaces. The {\bf product $ \mathcal G$-topology} on $X \times Y$ is the $ \mathcal G$-topology having as a basis the collection $\mathcal {B}$ of all sets of the form $U \times V$, where $U\in \mathcal G_X$ and $V\in \mathcal G_Y$. 
\end{defn}
%%%%%%%%%%%%%%%%%%%%%%%%%%%%
%%%%%%%%%%%%%
%%%%%%%%%                               G-topological groups
%%%%%%%%%%%%%%%%%%%%%%%%%
%%%%%%%%%%%%%%%%%%%%%%%%%%%%%%%%%%%%%%%%%%%%%%%%%%%%%%%%%%%%%%%%%
We are ready to give definition of $ \mathcal G$-topological groups.
\begin{defn}
A {\bf $ \mathcal G$-topological group} $G$ is a group that is also a $T_2$ $ \mathcal G$-topological space such that the multiplication map of $G \times G$ into $G$ sending $x \times y$ into $x \cdot y$, and the inverse map of $G$ into $G$ sending $x$ into $x^{-1}$, are $ \mathcal G$-continuous maps.
\end{defn}
\begin{example}
$(\R, +)$ is a group and it forms a $ \mathcal G$-topological space under the $ \mathcal G$-topology $ \mathcal G$ generated by the basis $\mathcal {B} = \{(-\infty, b), (a, \infty), a,b\in \R\}$. Then $((\R, +),\mathcal G)$ is a $ \mathcal G$-topological group.
\end{example}
We can define $\mathcal G$-topological subgroup using induced $\mathcal G$-topological subspace. 
\begin{defn}
Let $(X,\mathcal G)$ be a $\mathcal G$-topological space and let $Y$ be a subset of $X$.
Then 
$$
{\mathcal G}_Y =\{O\cap Y \, |\, O\in \mathcal G\}
$$
is $\mathcal G$-topology on $Y$. This $\mathcal G$-topology is called {\bf  subspace $\mathcal G$-topology induced by $\mathcal G$ into $Y$}. It is denoted by $(Y,{\mathcal G}_Y)$.
\end{defn}
\begin{prop}
{\bf i)} A $\mathcal G$-subspace of a $T_2$ $\mathcal G$-topological space is $T_2$ again {\bf ii)} Let $\varphi: X\rightarrow X'$ be a $\mathcal G$-continuous function and let $Y$ be a $\mathcal G$-subspace of $X$. Then the restriction map ${\varphi}|_Y : Y\rightarrow X'$ defined by ${\varphi}|_Y (y) = \varphi(y)$ is a $\mathcal G$-continuous function again.
\end{prop}
By last proposition we have the following result.
\begin{prop}
Any subgroup $H$ of a $\mathcal G$-topological group is a $\mathcal G$-topological group again, and is called {\bf $\mathcal G$-topological subgroup of $G$}.
\end{prop}
We can define morphisms of $ \mathcal G$-topological spaces.
\begin{defn}
Let $\varphi:G \rightarrow G'$ be a function. Then $\varphi$ is called a  {\bf morphism of $ \mathcal G$-topological groups} if $\varphi$ is both $ \mathcal G$-continuous and group homomorphism.

$\varphi$ is a {\bf $ \mathcal G$-topological group isomorphism} if it is
$\mathcal G$-homoemorphism and group homomorphism. 

If we have a $ \mathcal G$-topological group isomorphism between two $\mathcal G$-topological groups $G$ and $G'$ then we say that they are $\mathcal G$-isomorphic and we denote them by $G\cong_{\mathcal G} G'$.
\end{defn} 
It is obvious that composition of two $\mathcal G$-morphisms of $\mathcal G$-topological groups is again a $\mathcal G$-morphism. Also the identity map is a $\mathcal G$-isomorphism. So $\mathcal G$-topological groups and 
$\mathcal G$-morhisms form a category.

Now we will discuss homogenity of generalized structures.
\begin{defn}
A $\mathcal G$-topological space $X$ is said to be {\bf topologically homogeneous} if for any $x, y \in X,$ there is a $\mathcal G$-homeomorphism $\varphi: X\rightarrow X$ such that $\varphi(x)=y$.
\end{defn}
\begin{thm}
Let $G$ be a $ \mathcal G$-topological group and let $g\in G$.
Then left(right) translation map $L_g (R_g):G\rightarrow G$, defined by $L_g (x) = gx(R_g (x) = xg$), is a $\mathcal G$-topological homeomorphism.
\end{thm}
\begin{proof}
Here we will prove that $L_g$ is a $\mathcal G$-homeomorphism and similarly it can be shown that $R_g$ is a $\mathcal G$-homeomorphism.
First we will show that $L_g:G\rightarrow G$, defined by $L_g (x) =gx$, is $\mathcal G$-continuous. Since $L_g :G\rightarrow G$ is equal to the composition of 
$$
G \xrightarrow{i_g}G\times G \xrightarrow{m}G
$$
where $i_g (x) =(g,x)$ and $m$ is multiplicaion map in $G$. Then $L_g$ is $\mathcal G$-continuous because $i_g$ and $m$ are $\mathcal G$-continuous. Here we should verify that the map $i_g:G\rightarrow G\times G$ is $\mathcal G$-continuous. For any $\mathcal G$-open set $U\times V$, where $U,V$ are $\mathcal G$-open sets in $G$, ${i_g}^{-1}(U\times V ) = 
\begin{cases} V, & \,\text{if } \,  g \in U,\\ 
\emptyset,  &\, \text{if }\,  g \notin U,\end{cases}$ and we know that any $\mathcal G$-open set in product $\mathcal G$-topology of $G\times G$ can be written as a union of  $\mathcal G$-open sets of the form $U\times V $. Then $i_g$ is $\mathcal G$-continuous. Since $(L_g)^{-1} = L_{g^{-1}}$ is $\mathcal G$-continuous, the left translation map $L_g: G\rightarrow G$ is a $\mathcal G$-homeomorphism.
\end{proof}
Since, for any two points $g,g' \in G$, there exists a $\mathcal G$-homeomorphism $L_{g'g^{-1}} : G \rightarrow G$ such that $L_{g'g^{-1}(g)}=g'$, any $\mathcal G$-topological group is a topologically homogenous space.
\begin{defn}
Let $(X,\mathcal G)$ be a $ \mathcal G$-topological space and let $\mathcal {B} \subset \mathcal G$ be a base for $\mathcal G$. For $x\in X$, the family $$\mathcal {B}_x = \{O\in \mathcal {B}\, : \, x\in O\}\subset \mathcal {B} $$ is called  {\bf base at} $x$. 
\end{defn}
By left translation $L_g (O)= gO,\, O\in \mathcal {B}_e$, we have the following.
\begin{thm}
Let $G$ be a $\mathcal G$-topological group and let $e\in G$ be the unit element of $G$. For $g\in G$, the base at $g$ is equal to
$$\mathcal {B}_g = \{gO\, : \, O\in \mathcal {B}_e \}. $$ 
\end{thm}

\begin{thm}
Let $f: G\rightarrow H$ be a topological group homeomorphism. If $f$ is $\mathcal G$-continuous at the identity $e_G$ of $G$, then $f$ is $\mathcal G$-continuous at every $g\in G$. Hence $f$ is $\mathcal G$-continuous on $G$.
\end{thm}
\begin{proof}
Let $g\in G$ be any point. Suppose that $O$ is an $\mathcal G$-open neighbourhood of $h= f(g)$ in $H$. Since left translation $L_h$ is a $\mathcal G$-homeomorphism of $H$, there exists an $\mathcal G$-open neighbourhood $V$ of the identity element $e_H$ in $H$ such that $hV \subset O$. By $\mathcal G$-continuity of $f$ at $e_G$ we have a $\mathcal G$-open neighbourhood $U$ of $e_G$ in $G$ such that $f(U) \subset V$. Since $L_g$ is a $\mathcal G$-homeomorphism of $G$ onto itself, the set $gU$ is a $\mathcal G$-openn neighbourhood of $g$ in $G$, and we have that $f(gU)=hf(U)\subset hV\subset O$. Hence $f$ is $\mathcal G$-continuous at the point $g$. 
\end{proof}
\begin{thm}
Let $G$ be a $\mathcal G$-topological group and let $H$ be a subgroup of $G$. If $H$ contains a non-empty $\mathcal G$-open set then $H$ is $\mathcal G$-open in $G$.
\end{thm}
\begin{proof}
Let $U$ be a non-empty $\mathcal G$-open subset of $G$ with $U\subset H$. For every $h\in H,$ $L_h(U)= hU$ is $\mathcal G$-open in $G$. Therefore, the subgroup $H= \bigcup_{h\in H} hU$ is $\mathcal G$-open in $G$ by $Uh\subset H,\, \forall h \in H$.
\end{proof}
\begin{defn}
Let $X$ be a set and let $\Gamma \subset \mathcal P (X)$. We say that $\Gamma$ is a covering of $X$ if $X = \bigcup_{\gamma\in \Gamma} \gamma$. If $X$ is a $\mathcal G$-topological space and every element of $\Gamma$ is $\mathcal G$-open (or closed) then $\Gamma$ is called {\bf $\mathcal G$-open covering}(respectively {\bf $\mathcal G$-closed covering}).
\end{defn}
If a $\mathcal G$-topological space has a $\mathcal G$-open covering then it must be strong.
\begin{thm}\label{murad7}
Let $G$ be a $\mathcal G$-topological group and let $H$ be a subgroup of $G$. If $H$ is $\mathcal G$-open set then it is also $\mathcal G$-closed in $G$. 
\end{thm}
\begin{proof}
Let $\Gamma = \{gH: g \in G$ \} be the family of all left cosets of $H$ in $G$. This family is a disjoint $\mathcal G$-open covering of $G$ by left translations.  Therefore, every element of $\Gamma$, being compliment of union of all other elements, is $\mathcal G$-closed in $G$. In particular, $H=eH$ is $\mathcal G$-closed in $G$.
\end{proof}
\begin{thm}\label{murad4}
Let $G$ be a $\mathcal G$-topological group and $\mathcal B_e$ be a  base at the identity $e_G$ of $G$. Then we have the following properties.

{\bf i)}	For every $O \in \mathcal B_e$, there is an element $V \in \mathcal B_e$ such that $V^2 \subset O$.

\smallskip
{\bf ii)}	For every $O \in \mathcal B_e$, there is an element $V \in \mathcal B_e$ such that $V^{-1} \subset O$.

\smallskip
{\bf iii)}	For every $O \in \mathcal B_e$, and for every $x \in O$, there is an element $V \in \mathcal B_e$ such that $Vx \subset O$.

\smallskip
{\bf iv)}	For every $O \in \mathcal B_e$, and for every $x \in G$, there is an element $V \in \mathcal B_e$ such that $xVx^{-1} \subset O$.
\end{thm}
\begin{proof}
If $G$ is $\mathcal G$-topological group, then i) and ii) follow from the $\mathcal G$-continuity of the mappings $(x, y) \rightarrow xy$ and $x \rightarrow x^{-1}$ at the identity $e_G$. Property iii) follows from the $\mathcal G$-continuity of the left translation in $G$. Since $L_x, R_{x^{-1}}$ are $\mathcal G$-homoeomorphisms, their composition conjugation map $x \rightarrow axa^{-1}$ is also $\mathcal G$-homeomorphism. By this fact we have the property iv).
\end{proof}
\begin{defn}
Let $G$ be a $\mathcal G$-topological group. Then a subset $U$ of $G$ is called {\bf symmetric} if $U=U^{-1}$
\end{defn}
We can define $\mathcal G$-regular space.
\begin{defn}
Let $X$ be a $\mathcal G$-topological space. It is said to be {\bf $\mathcal G$-regular}, if for any $x \in X$ and any $\mathcal G$-closed set $F$ such that $x \notin F$, there are two $\mathcal G$-open sets $U$ and $V$ such that $x \in U$, $F \subset V$ and $U \cap V= \emptyset$.
\end{defn}
\begin{thm}
If a $\mathcal G$-topological group $G$ has a base at identity $e_G$ consisting of a symmetric neighbourhood, then it is a $\mathcal G$-regular space.
\end{thm}
\begin{proof}
Let $U$ be a $\mathcal G$-open set containing the identity $e_G$. By theorem \ref{murad4}, there is a $\mathcal G$-open set $V$ containing  $e_G$ such that $V^{-1}=V$ and $V^2 \subset U$. Then $x \in \cl_\mathcal G (V)$, and we have $Vx \cap V \neq \emptyset$. Hence $a_1 x=a_2$ for some $a_1$, $a_2$ in $V$. Thus $x={a_1}^{-1}a_2 \in V^{-1}V=V^2 \subset U$. It means $\cl_\mathcal G (V) \subset U$. Since $G$ is a homogeneous space, we get regularity of $G$.
\end{proof}
\begin{prop}
Let $X$ be a $T_2$ $\mathcal G$-topological space. Then for every $x\in X$ the singleton set $\{x\}$ is a $\mathcal G$-closed set in $X$.
\end{prop}
By the last result, the identity $\{e'\}$ is $\mathcal G$-closed and hence we have the following corollary.

\begin{cor}
Let $\varphi:G \rightarrow G'$ be a $\mathcal G$-morphism, let $e'$ be the unit element of $G'$. Then 
$$
\ker (\varphi)
=\{g\in G | \, \varphi(g) = e'\}
$$
is a $ \mathcal G$-closed normal $\mathcal G$-topological subgroup in $G$.
\end{cor}
From \cite{csaszar15} we have the following useful definition.
\begin{defn}
Let $X$ be a $\mathcal G$-topological space, and let $A\subset X$. We say that $x\in X$ is a point in $\mathcal G$-closure of $A$ if any $\mathcal G$-open set $U$ containing $x$ intersects $A.$  
\end{defn}

We will use the following result of \cite{Min11}:
\begin{lem}\label{murad1}
Let $(X, \mathcal G_X)$ and $(Y, \mathcal G_Y)$ be two $\mathcal G$-topological spaces. A function $f:X\rightarrow Y$ is $\mathcal G$-continuous $\Longleftrightarrow$  ${\cl}_{\mathcal G_X}(f^{-1}(B)) \subset  f^{-1}(\cl_{\mathcal G_Y}(B))$ for all $B \subset Y$ $\Longleftrightarrow$ $f(\cl_{\mathcal G_X}(A)) \subset  \cl_{\mathcal G_Y}(f(A)))$ for all $A \subset X$.
\end{lem}

We are ready to prove the following basic result.
\begin{thm}
Let $G$ be a $ \mathcal G$-topological group. Then 
$\mathcal G$-closure of any (normal) subgroup of $G$ is a $ \mathcal G$-topological (normal) subgroup again. 
\end{thm}
\begin{proof}
Let $H$ be a subgroup $G$. First we prove that ${\cl}_{\mathcal G}(H)$ is closed under multiplication $m$ in $G$. Given $x,y\in {\cl}_{\mathcal G}(H)$ and for any $\mathcal G$-open set $U$ containing $xy$, by last definition, we need show that 
$U \cap H \neq \emptyset$. Since $m:G\times G \rightarrow  G$ is $\mathcal G$-continuous, there exist $\mathcal G$-open sets $V$ and $W$ containing $x $ and $y$ respectively such that $m(V\times W) \subset U$. Since $x,y \in {\cl}_{\mathcal G}(H)$ then we have
$$
V\cap H \neq \emptyset \quad \text{and} \quad W\cap H \neq \emptyset.
$$
Hence $\emptyset \neq m(V\times W)\cap H \subset U\cap H$ which implies that $xy\in  {\cl}_{\mathcal G}(H)$.\\
Now ${\cl}_{\mathcal G}(H)$ is closed under the inverse operation because 
$(\cl_{\mathcal G}(H))^{-1} \subset 
\cl_{\mathcal G}(H^{-1})=\cl_{\mathcal G}(H)$ by Lemma \ref{murad1}.\\
Now suppose $H$ is normal subgroup in $G$. Given $g \in G$, let $\gamma_g:G \rightarrow G$ be conjugation by $g$. i.e. $\gamma_g (h) =ghg^{-1} = L_g \circ R_{g^{-1}} (h)$. Then $\gamma_g$ is a $\mathcal G$-homeomorphism from $G$ to $G$ itself. By Lemma \ref{murad1}, we have 
$$
\gamma_g (\cl_{\mathcal G}(H)) \subset \cl_{\mathcal G}(\gamma_g (H)) = \cl_{\mathcal G}(H), \quad \forall g\in G.
$$ 
It implies 
$$
\gamma_g (\cl_{\mathcal G}(H)) = g \cl_{\mathcal G}(H)) g^{-1} \subset \cl_{\mathcal G}(H)),  \quad \forall g\in G.
$$
Hence $\cl_{\mathcal G}(H))$ is a normal subgroup of $G$.
\end{proof}
\begin{thm}\label{murad6}
Let $G$ be a $\mathcal G$-topological group, $U$ a $\mathcal G$-open subset of $G$, and $A$ any subset of $G$. Then the set $AU$(respectively, $UA$) is $\mathcal G$-open in $G$.
\end{thm}
\begin{proof}
We know that every left translation of $G$ is a homeomorphism. Since, $AU = \bigcup_{a\in A} L_a(U)$, the conclusion follows. A similar argument follows for $UA$.
\end{proof}
\begin{thm}\label{murad5}
Let $G$ be a $\mathcal G$-topological group. Then for every subset $A$ of $G$ and every $\mathcal G$-open set $U$ containing the identity element $e_G$, $A\subset AU$($A \subset UA$).
\end{thm}
\begin{proof}
Since the inverse is $\mathcal G$-continuous, we can find a $\mathcal G$-open set $V$ containing $e-G$ such that $V^{-1}\subset U$. Take $x \in \cl_{\mathcal G}(A) $. Then $xV$ is a $\mathcal G$-open set containing $x$, therefore there is $a \in A \cap xV$, that is, $a=xb$, for some $b\in V$. Then $a=ab^{-1}\in AV^{-1}\subset AU$, hence, $\cl_{\mathcal G}(A) \subset AU$.
\end{proof}
\begin{thm}
Let $G$ be a $\mathcal G$-topological group, and $\mathcal B_e$ a base of the space $G$ at the identity element $e_G$. Then for every subset $A$ of $G$, $$\cl_{\mathcal G}(A)=\bigcap \{AU:U\in \mathcal B_e \}.$$
\end{thm}
\begin{proof}
In view of the theorem \ref{murad5}, we only have to verify that if $x$ is not in $\cl_{\mathcal G}(A)$, then there exists $U \in \mathcal B_e$, such that $x \notin AU$. Since $x \notin \cl_{\mathcal G}(A)$, then there exists a $\mathcal G$-open $W$ of $e_G$ such that $(xW) \cap A= \emptyset$. Take $U$ in $\mathcal B_e$, satisfying teh condition $U^{-1} \subset W$. Then $(xU^{-1}) \cap A = \emptyset$, which obviously implies that $AU$ does not conatin $x$.
\end{proof}
\begin{thm}
Let $G$ be a $\mathcal G$-topological group. Then $\mathcal G$-closure of a symmetric subset $A$ of $G$ is again symmetric.
\end{thm}
\begin{proof}
Since inverse mapping is a $\mathcal G$-homeomorphism, $\cl_{\mathcal G}(A^{-1})=(\cl_{\mathcal G}(A))^{-1} $ and $A$ is symmetric, $\cl_{\mathcal G}(A)=(\cl_{\mathcal G}(A))^{-1}$. Hence $\cl_{\mathcal G}(A)$ is symmetric.
\end{proof}
\maketitle

\section{$ \mathcal G$-connectedness for $\mathcal G$-Topological Groups}
In this section, we will discuss basic topological properties of connectivity of $\mathcal G$- topologicl spaces from \cite{csaszar14}.
\begin{defn}
Let $X$ be a $\mathcal G$-topological space and let $U,V \subset X$. Then we say that the pair $U,V$ is {\bf $\mathcal G$-separated} if ${\cl}_{\mathcal G}(U) \cap V = {\cl}_{\mathcal G}(V) \cap U =  \emptyset$.

Let $X$ be a $\mathcal G$-topological space. Then $X$ is {\bf $\mathcal G$-connected} if it can not be written as a union of two $\mathcal G$-separated sets.  
\end{defn}
\begin{cor}
A $\mathcal G$-topological group $G$, having an $\mathcal G$-open subgroup, is not $\mathcal G$-connected. 
\end{cor}
The connectedness for $\mathcal G$-topological subspaces can be defined by induced subspace $\mathcal G$-topology. From \cite{csaszar14}, we give the relation between $\mathcal G$-continuity and $\mathcal G$-connectedness.
\begin{thm}\label{murad3}
Let $f:X\rightarrow Y$ be a $\mathcal G$-continuous function between $\mathcal G$-topological spaces. If $X$ is $\mathcal G$-connected, so is $f(X)$.
\end{thm}
\begin{defn}
A function $f:X\rightarrow Y$ between $\mathcal G$-topological spaces is called {\bf $\mathcal G$-open} if for every $\mathcal G$-open set $U \subset X$, $f(U)$ is $\mathcal G$-open in $Y$.
\end{defn}
\begin{thm}
Let $f:X\rightarrow Y$ be a $\mathcal G$-open and injective function between $\mathcal G$-topological spaces and let $S\subset X$. If $f(S)$ is $\mathcal G$-connected, so is $S$.
\end{thm}

Now we will give definition of $\mathcal G$-component in $\mathcal G$-topological space.
\begin{defn}
Let $X$ be a $\mathcal G$-topological space, let $A\subset X$. For $x\in A$, the set 
$$
A_x =\bigcup_{x\in S\subset A}S,
$$
where $S$ is $\mathcal G$-connected in $A$, is called {\bf $\mathcal G$-component of $A$ containing $x$}.
\end{defn}
From \cite{csaszar14} we have the following theorem.
\begin{thm}\label{murad2}
The maximal component of a $\mathcal G$-closed set $A$ in $\mathcal G$-topological space $X$ is again $\mathcal G$-closed. 
\end{thm}
Hence the maximal component of  $\mathcal G$-topological space $X$ is again $\mathcal G$-closed.
Then we can give the following result.
\begin{thm}
Let $G$ be a $ \mathcal G$-topological group and let $e$ be the unit element of $G$. Then the maximal connected component containing $e$ is $\mathcal G$-closed normal subgroup of $G$.
\end{thm}
\begin{proof}
Let $F$ be the maximal component of the identity $e$. By Theorem \ref{murad2}, $F$ is $\mathcal G$-closed.

Let $a\in F.$ Since the multiplication and inversion functions in $G$ are $\mathcal G$-continuous,by \ref{murad3}, the set $aF^{-1}$ is also $\mathcal G$-connected, and since $e\in aF^{-1}$ we must have $aF^{-1}\subset F.$ Hence, for every $b\in F$ we have $ab^{-1}\in F,$ i.e. $F$ is a subgroup of $G$.

If $g$ is an arbitrary element of $G$, then $L_{g^{-1}}\circ R_g (F) = g^{-1}Fg$ is a $\mathcal G$-connected subset containing $e$. Since $F$ is maximal, $g^{-1}Fg \subset F$ for every $g\in G$, i.e. $F$ is a normal subgroup. 
\end{proof}
\begin{thm}
Let $G$ be a $\mathcal G$-connected $\mathcal G$-topological group, and let $U$ be a $\mathcal G$-open set containing identity element $e_G$ such that $U$ contains a symmetric $\mathcal G$-open set $V$ containing $e_G$. Then $G= \bigcup_{n=1}^\infty U^n$.
\end{thm}
\begin{proof}
By induction on $n$ and Theorem \ref{murad6}, for every positive integer $n$, $V^n$ is $\mathcal G$-open. Hence $\bigcup_{n=1}^\infty V^n$ is $\mathcal G$-open. Let $H= \bigcup_{n=1}^\infty V^n$. Since $V$ is symmetric, for every positive integer $n$, $V^n$ is symmetric. Hence $H=H^{-1}$. Also we have $V^kV^l=V^{(k+l)}$. Hence $HH=H$. So, $H$ is a subgroup of $G$. Since $H$ is a $\mathcal G$-open subgroup of $G$, by Theorem \ref{murad7}, $H$ is also $\mathcal G$-closed. Since $G$ is $\mathcal G$-connected, and $H$ is non-empty both $\mathcal G$-closed and $\mathcal G$-open, $G=H$. As $V \subset U$, it follows that $G= \bigcup_{n=1}^\infty U^n$.
\end{proof}

\end{document}